\documentclass{amsart}
\usepackage{enumerate}
\usepackage{geometry}
\usepackage{cleveref}
\usepackage{amsmath,amssymb,amsthm,mathrsfs,amsfonts}
\usepackage{url}
\usepackage{graphicx}
\usepackage{tikz}
\usetikzlibrary{shapes}
\usetikzlibrary{plotmarks}
\usetikzlibrary{arrows}
\usetikzlibrary{positioning}
\usepackage{pifont}

\newtheorem{thm}{Theorem}[section]

\newtheorem{lem}[thm]{Lemma}
\newtheorem{prop}[thm]{Proposition}
\newtheorem{cl}{Claim}

\theoremstyle{definition}
\newtheorem{defn}[thm]{Definition}

\newtheorem*{prob}{Problem}

\numberwithin{equation}{section}

\newcommand\myeq{\mathrel{\stackrel{\makebox[0pt]{\mbox{\normalfont\tiny def}}}{=}}}

\def\a#1{\mathfrak{#1}}
\def\u{\mathcal{U}}
\def\cu{\mathrm{K}}
\def\f{\a{Fr}_{m}\mathrm{K}}
\def\fz{\a{Fr}_{m-1}\mathrm{K}}

\begin{document}
\title{Atoms in infinite dimensional free sequence-set algebras}
\author{Mohamed Khaled}
\address{Mohamed Khaled, Faculty of Engineering and Natural Sciences, Bah\c{c}e\c{s}ehir University,
Istanbul, Turkey.}
\email{mohamed.khalifa@eng.bau.edu.tr}
\author{Istv\'an N\'emeti}
\address{Istv\'an N\'emeti, Alfr\'ed R\'enyi Institute of Mathematics, Hungarian Academy of Sciences, Budapest, Hungary.}
\email{nemeti.istvan@renyi.mta.hu}
\begin{abstract}
A. Tarski proved that the m-generated free algebra of $\mathrm{CA}_{\alpha}$, the class of cylindric algebras of dimension $\alpha$, contains exactly $2^m$ zero-dimensional atoms, when $m\ge 1$ is a finite cardinal and $\alpha$ is an arbitrary ordinal. He conjectured that, when $\alpha$ is infinite, there are no more atoms. This conjecture has not been confirmed or denied yet. In this article, we show that Tarski's conjecture is true if $\mathrm{CA}_{\alpha}$ is replaced by $\mathrm{D}_{\alpha}$, $\mathrm{G}_{\alpha}$, but the $m$-generated free $\mathrm{Crs}_{\alpha}$ algebra is atomless.
\end{abstract}
\maketitle
\section{Introduction}
Free algebras play an important role in universal algebra, and especially in the theory of Boolean algebras with operators (BAO's), see, e.g., \cite{andjonnem}, \cite{nem86}, \cite[\S 5.6]{hmt2} and \cite{madnem}. One of the first things to investigate about these free algebras is whether they are atomic or not, i.e., whether their Boolean reducts are atomic or not. An atomic Boolean algebra is an algebra in which below every non-zero element there is an atom, i.e., a minimal non-zero element. Let $\mathrm{K}$ be a class of similar algebras. For each cardinal $m$, let $\mathfrak{Fr}_m\mathrm{K}$ stand for the $m$-generated free $\mathrm{K}$ algebra. 

Cylindric algebras are special BAO's that were introduced by A. Tarski around 1947. These are Boolean algebras equipped with unary operations, called cylindrifications, and constant symbols, called diagonals. These algebras capture the intrinsic algebraic side of first order logic (FOL), see \cite[section 4.3]{hmt1}. Let $\alpha$ be any ordinal.

\begin{defn}\label{ca}
A cylindric algebra of dimension $\alpha$ is an algebra of the form $$\a{A}=\langle A, +,\cdot, -,0,1,c_i,d_{ij}\rangle_{i,j\in\alpha},$$
where $A$ is a non-empty set, $+,\cdot$ are binary operations, $-,c_i$ are unary operations, $0,1,d_{ij}$ are constant symbols, and $\a{A}$ satisfies the following postulates for every $x,y\in A$ and every $i,j,k\in\alpha$:
\begin{enumerate}
\setcounter{enumi}{-1}
\renewcommand{\theenumi}{(CA \arabic{enumi})}
\renewcommand{\labelenumi}{\theenumi}
\item\label{ca0} $\langle A,+,\cdot,-,0,1\rangle$ is Boolean algebra,
\item\label{ca1} $c_i 0=0$,
\item\label{ca2} $x+c_ix=c_ix$,
\item\label{ca3} $c_i(x\cdot c_iy)=c_ix\cdot c_iy$,
\item\label{ca4} $c_ic_jx=c_jc_ix$,
\item\label{ca5} $d_{ii}=1$,
\item\label{ca6} If $k\not=i,j$, then $d_{ij}=c_k(d_{ik}\cdot d_{kj})$,
\item\label{ca7} If $i\not=j$, then $c_i(d_{ij}\cdot x)\cdot c_i(d_{ij}\cdot -x)=0$.
\end{enumerate}
\end{defn}
The class of all cylindric algebras of dimension $\alpha$ is denoted by $\mathrm{CA}_{\alpha}$. Atoms in the free cylindric algebras correspond to finitely axiomatizable complete and consistent theories of FOL. In the present paper, we are interested in the case when $\alpha$ is infinite, so $\alpha\ge\omega$ is assumed throughout the paper.

We will prove some results connected to a conjecture of A. Tarski \cite[Remark 2.5.12 and Problem 2.6]{hmt1}. This conjecture is concerned with atoms and zero-dimensional elements in the finitely generated free cylindric algebras of dimension $\alpha$. An element $a$ of a cylindric algebra $\mathfrak{A}$ is said to be  zero-dimensional if it is a fixed-point of all the cylindrifications, i.e., if $c_ia=a$ for each $i\in\alpha$.

In this section, we outline the background for this conjecture and we state our main theorems. Let $m$ be any cardinal. Let $\mathrm{K}\subseteq \mathrm{CA}_{\alpha}$ be a class of cylindric algebras containing at least one non-trivial algebra (i.e., having more than one-element). The following are true:

\begin{enumerate}
\renewcommand{\theenumi}{(Fact \arabic{enumi})}
\renewcommand{\labelenumi}{\theenumi}
\item If $m$ is infinite, then $\mathfrak{Fr}_m\mathrm{K}$ is atomless. This is due to D. Pigozzi \cite[Theorem 2.5.13]{hmt1}, and the proof can be easily generalized for any class of Boolean algebras with operators.
\item\label{fact2} Assume that $m$ is finite. In \cite[Theorem 2.5.11]{hmt1}, it is proved that $\mathfrak{Fr}_m\mathrm{K}$ contains exactly $2^m$ zero-dimensional atoms
\footnote{Tarski's proof is only about $\mathrm{CA}_{\alpha}$, but it works verbatim for the class $\mathrm{K}$.}.
\item\label{fact3} The $0$-generated free algebra $\mathfrak{Fr}_0\mathrm{K}$ contains exactly one atom, namely $–c_0-d_{01}$, by \cite[Theorem 2.5.11]{hmt1}. We show that it contains no other atom. Indeed, let $a\in\mathfrak{Fr}_0\mathrm{K}$ be such that $a\cdot c_0-d_{01}\not=0$. It is not hard to see that, for every $x\in\mathfrak{Fr}_0\mathrm{K}$, the set $\{i\in\alpha:c_ix\not=x\}$ is finite (which means $\mathfrak{Fr}_0\mathrm{K}$ is locally finite dimensional, see \cite[Definition~1.11.1 (i)]{hmt1}). This is true because $\mathfrak{Fr}_0\mathrm{K}$ is $0$-generated. So, we can find $i,j\in\alpha$ such that $i\not=j$, $c_ia=a$ and $c_ja=a$. By \cite[Theorem 1.3.19]{hmt1}, it follows that $a\cdot c_i-d_{ij}\not=0$. Now, by the cylindric equations \ref{ca0} - \ref{ca7}, we have 
$$c_i(a\cdot d_{ij})=c_i(c_ia\cdot d_{ij})=c_ia\cdot c_id_{ij}=a\cdot 1=a\not=0,$$
$$c_i(a\cdot -d_{ij})=c_i(c_ia\cdot -d_{ij})=c_ia\cdot c_i-d_{ij}=a\cdot c_i-d_{ij}\not=0.$$
Hence, $a\cdot d_{ij}\not=0$ and $a\cdot -d_{ij}\not=0$. Thus, $a$ is not an atom as desired.

\item\label{fact4} Suppose that $m\geq 1$. Tarski conjectured that all the atoms in the free algebra $\mathfrak{Fr}_m\mathrm{CA}_{\alpha}$ are zero-dimensional. See \cite[Remark 2.5.12 and Problem 2.6]{hmt1}. To the best of our knowledge, this conjecture remains open.
\end{enumerate} 

$\mathrm{Gs}_{\alpha}$ denotes the class of $\alpha$-dimensional representable cylindric algebras, it will be defined in the next section. First, we show that Tarski's conjecture is true when $\mathrm{CA}_{\alpha}$ is replaced by $\mathrm{Gs}_{\alpha}$.

\begin{thm}\label{nemeti} For each finite cardinal $m$, there are exactly $2^m$ many atoms in $\mathfrak{Fr}_m\mathrm{Gs}_{\alpha}$, each of these atoms is zero-dimensional. 
\end{thm}
\begin{proof}
Recall from \ref{fact2} that the free algebra $\mathfrak{Fr}_m\mathrm{Gs}_{\alpha}$ contains exactly $2^m$ many zero-dimensional atoms. These atoms are listed in Lemma~\ref{lem2} herein (and also in \cite[Theorem 2.5.11]{hmt1}), and it is apparent that their sum is $-c_0-d_{01}$. It remains to prove that there is no atom below $c_0-d_{01}$ in the algebra $\mathfrak{Fr}_m\mathrm{Gs}_{\alpha}$. We use the fact that $\mathrm{Gs}_{\alpha}$ is generated as a variety by its locally finite dimensional algebras, see \cite[Theorem 3.1.123]{hmt2}. Assume that $a\in\mathfrak{Fr}_m\mathrm{Gs}_{\alpha}$ is such that $a\cdot c_0-d_{01}\not=0$. We show that $a$ is not an atom. We may assume that $a$ is a term, so it has a value in any algebra and evaluation pair. By  $a\ne 0$  in $\mathfrak{Fr}_m\mathrm{Gs}_{\alpha}$ we have that $a \ne 0$ in some locally finite dimensional algebra  $\mathfrak{A}\in\mathrm{Gs}_{\alpha}$ (with an appropriate evaluation of the variables occurring in  $a$). In the algebra $\mathfrak{A}$, we can find $i,j\in\alpha$ such that $i\not=j$, $c_ia=a$ and $c_ja=a$. Recall $a\cdot c_0-d_{01}\not=0$, so \cite[Theorem 1.3.18 (iii)]{hmt1} implies that $a\cdot c_i-d_{ij}\not=0$ is also true in $\mathfrak{A}$. Again, similarly to our argument in (Fact 3),
$$\mathfrak{A}\models c_i(a\cdot d_{ij})=c_i(c_ia\cdot d_{ij})=c_ia\cdot c_id_{ij}=a\cdot 1=a\not=0,$$
$$\mathfrak{A}\models c_i(a\cdot -d_{ij})=c_i(c_ia\cdot -d_{ij})=c_ia\cdot c_i-d_{ij}=a\cdot c_i-d_{ij}\not=0.$$
Hence, $a\cdot d_{ij}\not=0$ and $a\cdot -d_{ij}\not=0$ are true in $\mathfrak{A}$. Thus, the same is true in the free algebra  $\mathfrak{Fr}_m\mathrm{Gs}_{\alpha}$, so $a$ is not an atom, as desired.
\end{proof}

In this article, we investigate whether the above theorem remains true if $\mathrm{Gs}_{\alpha}$ is replaced by any of the relativized classes of cylindric algebras $\mathrm{Crs}_{\alpha}$, $\mathrm{D}_{\alpha}$, $\mathrm{G}_{\alpha}$. These classes will be defined in the next section. The notion of a relativized algebra was introduced in algebraic logic by L. Henkin. Relativization was proved potent in obtaining positive results in logic, see, e.g., \cite{marx2002}, \cite{mikulas98} and \cite{nem95}. For some properties of these classes, see, e.g., \cite{And01}, \cite{And88} and \cite{nem85}. For instance, in contrary to $\mathrm{Gs}_{\alpha}$, the equational theories of the classes $\mathrm{Crs}_{\alpha}$ and $\mathrm{G}_{\alpha}$ are decidable \cite{nem86,nem95}. The decidability of the equational theory of the class $\mathrm{D}_{\alpha}$ remains open.

\begin{thm}[Main Result 1]\label{main} Let $\alpha \geq\omega$ be an infinite ordinal and let $m\geq 1$ be a finite cardinal. The following are true:
\begin{enumerate}
\renewcommand{\theenumi}{(\arabic{enumi})}
\renewcommand{\labelenumi}{\theenumi}
\item The free algebra $\a{Fr}_{m-1}\mathrm{Crs}_{\alpha}$ is atomless.
\item The free algebra $\a{Fr}_m\mathrm{K}$ is not atomic, but it contains exactly $2^m$-many atoms each of which is zero-dimensional, when $\mathrm{K}$ is any of $\mathrm{D}_{\alpha}$, $\mathrm{G}_{\alpha}$ or $\mathrm{Gs}_{\alpha}$.
\end{enumerate}
\end{thm}

The proof of Theorem~\ref{nemeti} depends essentially on the fact that $\mathrm{Gs}_{\alpha}$ is generated by locally finite dimensional algebras. The same is not true for the relativized classes of cylindric algebras, therefore the same argument cannot be used to prove Theorem~\ref{main}. We will prove a stronger theorem, Theorem~\ref{realmain} in the next section, which will imply Theorem~\ref{main}. 

The proofs of Theorem~\ref{main} and Theorem~\ref{realmain} go by showing that there are no elements in the free algebras that are disjoint from all the diagonals $d_{ij}$. Theorem~\ref{nemadd} below shows that the same is not true for $\mathrm{CA}_{\alpha}$, therefore for settling the conjecture for $\mathrm{CA}_{\alpha}$, one has to use other techniques, too.

\begin{thm}[Main Result 2]\label{nemadd}
Let $\alpha\ge 2$ be any ordinal and let $m\geq 1$ be a finite cardinal. Then, there is $x\in\a{Fr}_m\mathrm{CA}_{\alpha}$ such that $x\not=0$ and $x\leq -d_{ij}$ for every $ i,j\in \alpha\sim 2$ with $i\not=j$.
\end{thm}
It is worth of note that atomicity of these free algebras correspond to the failure of a version of G\"odel's incompleteness theorem for the corresponding logics. For more detail about this correspondence, see \cite{nem86}, \cite{zalan} and \cite{myphd}. For results concerning the atomicity of free algebras of logics, one can see \cite{hmt1,hmt2,nem86,andjonnem,monk,madnem,zalan,andnem2013,myphd,myigpl,apal,aim}.
\section{Algebras of sets of sequences}
Throughout, fix an infinite ordinal $\alpha$. A function with domain $\alpha$ is called a \emph{\textbf{sequence} of length $\alpha$} (a \emph{sequence} for short). For every $i\in \alpha$ and every two sequences $f,g$, we write $f\equiv_ig$ if and only if $g= f(i/u)$ for some $u$, where $f(i/u)$ is the sequence which agrees with $f$ everywhere except that it's value at $i$ equals $u$. Let $V$ be a set of sequences of length $\alpha$. Such set is called an \emph{$\alpha$-dimensional \textbf{unit}}. The smallest set $U$ that satisfies $V\subseteq{^{\alpha}U}$ is called \emph{the \textbf{base} of the unit $V$}.

Let $i,j\in\alpha$. Define the \emph{$ij$-\textbf{diagonal} of the unit $V$} as follows: $D_{ij}^{[V]}=\{f\in V:f(i)=f(j)\}$. For each $X\subseteq V$, define $C_i^{[V]}X=\{f\in V:(\exists g\in X) \ f\equiv_ig\}$. This is called the \emph{$V$-\textbf{cylindrification} of $X$ in the direction $i$}. When no confusion is likely, we omit the superscript $[V]$ from the above defined objects.
\begin{defn}
The \emph{\textbf{full cylindric-like algebra} over the unit $V$} is an algebra of the form $$\a{P}(V)\myeq\langle \mathcal{P}(V), \cup, \cap, \sim, \emptyset, V, C_i, D_{ij}\rangle_{i,j\in\alpha},$$
where $\mathcal{P}(V)$ is the family of all subsets of $V$, $\cup,\cap,\sim$ are the Boolean set-theoretic operations, $\emptyset$ is the empty set, and the $C_i$'s and the $D_{ij}$'s are as defined above.
\end{defn}
Let $\mathrm{K}$ be a class of algebras of same signature. Then, $\mathbf{I}\mathrm{K}$, $\mathbf{S}\mathrm{K}$ and $\mathbf{H}\mathrm{K}$ are the classes that consist of the isomorphic copies, subalgebras and homomorphic images, respectively, of the members of $\mathrm{K}$.

\begin{defn} We define the following classes of cylindric-like set algebras.
\begin{itemize}
\item The class of all \emph{\textbf{relativized} cylindric set algebras} is given by $$\mathrm{Crs}_{\alpha}=\mathbf{IS}\{\a{P}(V):V\text{ is an $\alpha$-dimensional unit}\}.$$
\item The class of \emph{\textbf{diagonalizable} cylindric set algebras} $\mathrm{D}_{\alpha}$ is given by $$\mathrm{D}_{\alpha}=\mathbf{IS}\{\a{P}(V):V\text{ is a diagonalizable $\alpha$-dimensional unit}\},$$
where an $\alpha$-dimensional unit $V$ is diagonalizable if it satisfies $f(i/f(j))\in V$, for each $f\in V$ and each $i,j\in\alpha$.
\item The class of \emph{\textbf{locally square} cylindric set algebras} $\mathrm{G}_{\alpha}$ is given by $$\mathrm{G}_{\alpha}=\mathbf{IS}\{\a{P}(V):V\text{ is a union of $\alpha$-dimensional squares}\},$$
where a union of $\alpha$-dimensional squares is an $\alpha$-dimensional unit of the form $V=\bigcup_{i\in I}{^{\alpha}U_i}$ for some family of non-empty sets $\{U_i:i\in I\}$.
\item The class of \emph{\textbf{generalized} cylindric set algebras} $\mathrm{Gs}_{\alpha}$ is given by $$\mathrm{Gs}_{\alpha}=\mathbf{IS}\{\a{P}(V):V\text{ is a union of mutually disjoint $\alpha$-dimensional squares}\},$$
where $V$ is a union of mutually disjoint $\alpha$-dimensional squares if there is a family of mutually disjoint non-empty sets $\{U_i:i\in I\}$ such that $V=\bigcup_{i\in I}{^{\alpha}U_i}$.
\end{itemize}
\end{defn}

We note that $\mathrm{Crs}_{\alpha}$, $\mathrm{D}_{\alpha}$, $\mathrm{Gs}_{\alpha}$ and $\mathbf{H}\mathrm{G}_{\alpha}$ are varieties, and it is still open whether $\mathrm{G}_{\alpha}=\mathbf{H}\mathrm{G}_{\alpha}$. Since we are dealing with many classes, it is more convenient to prove a general theorem which implies Theorem~\ref{main}. We need to generalize our definitions, too.
\begin{defn}
Let $\u$ be a class of $\alpha$-dimensional units. 
\begin{itemize}
\item We say that $\u$ \emph{\textbf{supports diagonalization}} iff $V\cup\{f(i/f(j))\}\in\u$ for each $V\in\u$, each $f\in V$ and each $i,j\in\alpha$.
\item We say that $\u$ \emph{\textbf{requires diagonalization}} iff $\u$ contains a singleton, and $f(i/f(j))\in V$ for each $V\in\u$, each $f\in V$ and each $i,j\in \alpha$.
\end{itemize}
\end{defn}
For any class $\u$ of $\alpha$-dimensional units, if $\u$ requires diagonalization then $\u$ must also support diagonalization, but the converse is not necessarily true. Let $\mathrm{K}$ be a class of similar algebras, then $\mathbf{V}(\mathrm{K})$ stands for the smallest variety containing $\mathrm{K}$. One can easily see that each of the classes $\mathrm{Gs}_{\alpha}$, $\mathrm{D}_{\alpha}$ and $\mathbf{H}\mathrm{G}_{\alpha}$ can be viewed as $\mathbf{V}(\{\mathfrak{P}(V):V\in\u\})$ for  some class of $\alpha$-dimensional units $\u$ that require diagonalization. The same is not true for the class $\mathrm{Crs}_{\alpha}$. However, $\mathrm{Crs}_{\alpha}$ can be introduced as the variety generated by the class of full algebras of all $\alpha$-dimensional units.

\begin{thm}\label{realmain}
Let $\mathrm{K}$ be a class of cylindric-type algebras such that $\mathbf{V}(\mathrm{K})=\mathbf{V}(\{\a{P}(V):V\in\u\})$ for some class $\u$ of $\alpha$-dimensional units. Let $m\geq 1$ be a finite cardinal. The following are true:
\begin{enumerate}
\renewcommand{\theenumi}{(\arabic{enumi})}
\renewcommand{\labelenumi}{\theenumi}
\item If $\u$ supports diagonalization, then the free algebra $\f$ is not atomic.
\item If $\u$ requires diagonalization, then the following are true:
\begin{enumerate}
\renewcommand{\theenumi}{(\alph{enumi})}
\renewcommand{\labelenumi}{\theenumi}
\item $\f$ contains exactly $2^m$-many atoms.
\item All the atoms of $\f$ are zero-dimensional.
\item\label{itemc} There is a decomposition $\f\cong\a{A}\times\a{B}$ such that $\vert A\vert=2^{2^{m}}$, $\a{A}$ is discrete and $\a{B}$ is atomless. 
\end{enumerate}
\item[(3)] If $\u$ is the class of all $\alpha$-dimensional units, then $\fz$ is atomless.
\end{enumerate}
\end{thm}

Theorem 2.3 implies Theorem 1.2 since the classes of $\mathrm{D}_{\alpha}$-units, $\mathrm{G}_{\alpha}$-units and $\mathrm{Gs}_{\alpha}$-units require diagonalization, while the class of $\mathrm{Crs}_{\alpha}$-units supports diagonalization. We divide the proof of Theorem~\ref{realmain} into some lemmas and propositions. Throughout the remaining part of this paper, fix classes $\mathrm{K}$ and $\u$, and a cardinal $m\geq 1$ satisfying the assumptions of the Theorem~\ref{realmain}. 
\section{The atomic part in $\f$}
Let $ct_{\alpha}$ be \emph{the \textbf{algebraic type} of cylindric-like algebras}, it consists of binary operations $+$, $\cdot$, unary operations $-,c_i$ ($i\in\alpha$) and constant symbols $0,1,d_{ij}$ ($i,j\in\alpha$). Let $Y$ be any set, \emph{the \textbf{set of all terms $T_{\alpha}(Y)$} generated by $Y$ in the signature $ct_{\alpha}$} is defined to be the smallest set satisfying:
\begin{itemize}
\item $Y\subseteq T_{\alpha}(Y)$ and $0,1,d_{ij}\in T_{\alpha}(Y)$ for each $i,j\in\alpha$,
\item For each $\tau,\sigma\in T_{\alpha}(Y)$, we have $\tau+\sigma,\tau\cdot\sigma,-\tau\in T_{\alpha}(Y)$,
\item For each $\tau\in T_{\alpha}(Y)$ and each $i\in\alpha$, we have $c_i\tau\in T_{\alpha}(Y)$.
\end{itemize}

Note that the equational theory of $\cu$ coincides with the equational theory of $\{\a{P}(V):V\in\u\}$. So, for example, whenever we say that $\cu\not\models\tau=0$, for some term $\tau\in T_{\alpha}(Y)$, we will assume that there is a unit $V\in\u$, $f\in V$ and an evaluation $\iota:Y\rightarrow\mathcal{P}(V)$ such that $(V,f,\iota)\models\tau$. The latter means that $f$ is a member of the interpretation of $\tau$ in the algebra $\a{P}(V)$ under the evaluation $\iota$. Examples of equations that are true inƒ the class $\mathrm{K}$ (cf., \cite[Theorem 9.4]{monk} and \cite[Theorem 1.2.6 (ii) and Theorem 1.2.11]{hmt1}): For $i,j\in\alpha$ with $i\not=j$,
\begin{enumerate}
\renewcommand{\theenumi}{(Eq \arabic{enumi})}
\renewcommand{\labelenumi}{\theenumi}
\item\label{eqc1} $c_i 0=0$.
\item\label{eqc2} $x\cdot c_ix=x$.
\item\label{eqc3} $c_i(x\cdot c_iy)=c_ix\cdot c_iy$.
\item\label{eqc4} $c_i(x+y)=c_ix+c_iy$.
\item\label{eqc5} $c_i(-c_ix)=-c_ix$.
\item\label{eqc6} $d_{ii}=1$.
\item\label{eqc7} $c_i(x\cdot d_{ij})\cdot d_{ij}=x\cdot d_{ij}$.
\end{enumerate}

Let $X=\{x_0,\ldots,x_{m-1}\}$ be the generating set of the free algebra $\f$. Let $q\in{^X\{-1,1\}}$, such function is called \emph{a \textbf{choice function} for $X$}. For each $x_k\in X$, let $x_k^q=x_k$ if $q(x_k)=1$ otherwise let $x_k^q=-x_k$. Define $X^q=x_0^q \cdots x_{m-1}^q$.
\begin{lem}\label{lem}
Suppose that $\u$ requires diagonalization. Let $i,j\in\alpha$ be such that $i\not=j$, and let $q\in{^X\{-1,1\}}$ be a choice function. Then, $$\cu\models X^{q}\cdot -c_0-d_{01}=X^{q}\cdot -c_i-d_{ij}.$$ 
Consequently, $X^{q}\cdot -c_0-d_{01}$ is a zero-dimensional element in the algebra $\f$.
\end{lem}
\begin{proof}
Suppose that $\u$, $i$, $j$ and $q$ are as required above. Let $V\in\u$, $f\in V$ and $\iota:X\rightarrow\mathcal{P}(V)$ be such that $(V,f,\iota)\models X^{q}\cdot -c_i-d_{ij}$. Then $f(i)=f(j)$. Suppose that $f(0)\not=f(i)$. Since $\u$ requires diagonalization then $f(i/f(0))\in V$, and $(V,f(i/f(0)),\iota)\models -d_{ij}$. This implies that $(V,f,\iota)\models c_i-d_{ij}$ which contradicts the assumptions. Hence, $f(0)=f(i)$ and similarly $f(1)=f(i)$. Now, we show that $(V,f,\iota)\models -c_0-d_{01}$. Suppose towards a contradiction that $(V,f,\iota)\models c_0-d_{01}$. Then there is $u$ in the base of $V$ and $g=f(0/u)$ such that $(V,g,\iota)\models -d_{01}$. Hence, $u\not=f(i)$. By assumptions, we have $g_1=g(i/g(0))\in V$ and $g_2=g_1(0/g(j))\in V$. Then $g_2=f(i/u)$, which implies $(V,f,\iota)\models c_i-d_{ij}$. This contradicts the assumptions. Thus, $(V,f,\iota)\models -c_0-d_{01}$. We have shown that $\cu\models X^{q}\cdot -c_i-d_{ij}\leq X^{q}\cdot-c_0-d_{01}$. The desired follows by the symmetry of indices. 

Let $\tau\myeq X^{q}\cdot-c_0-d_{01}$. To show that $\tau$ is zero-dimensional, we need to prove that $c_i\tau=\tau$. By the first part, we have

$$c_i\tau=c_i(X^{q}\cdot -c_i-d_{ij})=c_iX^q\cdot -c_i-d_{ij}\leq -c_i-d_{ij}\leq d_{ij}.$$
Thus,
\begin{eqnarray*}
\tau&=&d_{ij}\cdot \tau\\
&=& d_{ij}\cdot c_i(d_ {ij}\cdot \tau) \hspace{2cm} \text{by (Eq 7)}\\
&=& d_{ij}\cdot c_i\tau\\
&=& c_i\tau.
\end{eqnarray*}
Hence, $\tau=X^{q}\cdot-c_0-d_{01}$ is zero-dimensional and we are done.
\end{proof}
Now, we will show that each of the zero-dimensional elements, given in the above lemma, is an atom in the free algebra $\f$.
\begin{lem}\label{lem2}
Suppose that $\u$ requires diagonalization. Let $i,j\in\alpha$ be such that $i\not=j$, and let $q\in{^X\{-1,1\}}$ be a choice function. Then $X^{q}\cdot -c_0-d_{01}$ is an atom in the free algebra $\f$.
\end{lem}
\begin{proof}
Suppose that $\u$ requires diagonalization. Let $i,j\in\alpha$ be such that $i\not=j$, and let $q\in{^X\{-1,1\}}$ be a choice function. Let $\tau\myeq X^{q}\cdot -c_0-d_{01}$. Let $\{w\}\in\u$, such unit is guaranteed to exist by the assumption that $\u$ requires diagonalization. Define $\nu:X\rightarrow\{\emptyset,\{u\}\}$ as follows: For each $x_k\in X$, let $\nu(x_k)=\{w\}$ if $q(x_k)=1$ and $\nu(x_k)=\emptyset$ otherwise. Clearly, $(\{w\},w,\nu)\models\tau$, i.e., $\tau$ is not zero in $\f$. To prove that $\tau$ is an atom in $\f$, it is enough to prove the following: For any term $\sigma\in T_{\alpha}(X)$,
\begin{equation}\label{atom}
\text{either } \ \cu\models\tau\cdot\sigma=0 \ \text{ or } \ \cu\models\tau\cdot-\sigma=0.
\end{equation}
We prove \eqref{atom} by induction on the complexity of the term $\sigma$. Obviously, \eqref{atom} holds if $\sigma=x_k$ for some $x_k\in X$. Also,  Lemma~\ref{lem} guarantees that \eqref{atom} is true if $\sigma=d_{ij}$ for some $i,j\in\alpha$. It is not hard to see that the induction step holds for the Boolean connectives. We will show that the induction step also holds for the cylindrification operations. To do that, we will use that fact that cylinderifications are additive and complemented operators. Let $\sigma=c_k\sigma'$ for some $\sigma'\in T_{\alpha}(X)$ and $k\in\alpha$. Remember that $\tau$ is zero-dimensional, so $\cu\models c_k\tau=\tau$ and $\cu\models c_k-\tau=\tau$. By the induction hypothesis, we have one of the following cases.
\begin{enumerate}
\renewcommand{\theenumi}{(\alph{enumi})}
\renewcommand{\labelenumi}{\theenumi}
\item Either, $\cu\models\tau\cdot\sigma'=0$. In this case, we have 
$$\cu\models\tau\cdot\sigma=\tau\cdot c_k\sigma'=c_k\tau\cdot c_k\sigma'=c_k(\tau\cdot\sigma')=0.$$
\item Or, $\cu\models\tau\cdot-\sigma'=0$. Here, 
$$\cu\models \tau\cdot-\sigma=-(-\tau+\sigma)=-(c_k-\tau+c_k\sigma')=-c_k(-\tau+\sigma')\leq -(-\tau+\sigma')=0.$$
\end{enumerate}
We have proved \eqref{atom}. Therefore, $\tau=X^{q}\cdot -c_0-d_{01}$ is an atom in $\f$.
\end{proof}
We showed that $\f$ contains at least $2^m$ zero-dimensional atoms if $\u$ requires diagonalization. Note that, in this case, the sum of all these atoms in $\f$ is equal to $-c_0-d_{01}$ which can be shown to be zero-dimensional element by the argument used in Lemma~\ref{lem}. In the following section, we will prove that $\f$ does not contain any extra atom. 
\section{The non-atomic part in $\f$}
For each term $\sigma\in T_{\alpha}(X)$, we let $index(\sigma)$ be the set of all indices $i\in\alpha$ that appear in $\sigma$. Let $\Gamma\subseteq\alpha$ and let $f,g$ be two sequences of length $\alpha$. We write $f\equiv_{\Gamma}g$ if and only if $f(k)=g(k)$ for each $k\in\alpha\sim\Gamma$. We start with the following proposition.
\begin{prop}\label{prop} Suppose that $\u$ supports diagonalization. Then there is no atom below $c_0-d_{01}$ in the free algebra $\f$.
\end{prop}
\begin{proof}
Suppose that $\u$ supports diagonalization. Let $\tau\in T_{\alpha}(X)$ be a cylindric-term that satisfies $\cu\not\models\tau\cdot c_0-d_{01}=0$. We prove that $\tau\cdot c_0-d_{01}$ is not an atom in $\f$. Let $V\in\u$, $f\in V$ and $\iota$ be an evaluation such that $(V,f,\iota)\models\tau\cdot c_0-d_{01}$. We can find $g\in V$ such that $g=f(0/u)$, for some $u\not=f(1)$, and $(V,g,\iota)\models-d_{01}$. Let $\Gamma=index(\tau)\cup\{0,1\}$, since every term is built up from finitely many symbols in the signature $ct_{\alpha}$ then $\Gamma$ must be finite. Let $i,j\in\alpha\sim\Gamma$ be such that $i\not=j$.

\noindent{\textbf{Case 1:}} Suppose that $g(i)=g(j)$. Recall that $g(0)\not=g(1)$. So, without loss of generality,  we may assume that $g(0)\not=g(j)$. Let 
$$W=\{h\in V: h(i)=h(j)\} \ \ \text{ and } \ \ V'=V\cup\{g(i/g(0))\}.$$
Note that $V'\in\u$ because $\u$ supports diagonalization. Define $\iota_1,\iota_2:X\rightarrow\mathcal{P}(V')$ as follows: For each $x_k\in X$,
$$\iota_1(x_k)=\iota(x_k)\cap W \ \ \text{ and } \ \ \iota_2(x_k)=\iota_1(x_k)\cup\{g(i/g(0))\}.$$
For each $\sigma\in T_{\alpha}(X)$ and each $h\in V$, if $index(\sigma)\subseteq \Gamma$ and $h\equiv_{\Gamma}g$, then
\begin{equation}\label{B12}
(V',h,\iota_1)\models\sigma\iff(V,h,\iota)\models\sigma\iff(V',h,\iota_2)\models\sigma.
\end{equation}
This can be shown by a simple induction argument on the complexity of the term $\sigma$ as follows. Obviously, \eqref{B12} is true for the case when $\sigma=x_k\in X$ and when $\sigma=d_{k\lambda}$, $k,\lambda\in\Gamma$. Also, it is not hard to see that the induction step holds for the Boolean connectives. Now, suppose that $\sigma=c_k\sigma'$ and $index(\sigma)\subseteq\Gamma$. That means $k\in\Gamma$ and $index(\sigma')\subseteq\Gamma$, too. Let $h\in V$ be such that $h\equiv_{\Gamma}g$. Note that for any $h'\in V'$, if $h\equiv_k h'$ then $h'\equiv_{\Gamma}g$ and so $h'\in V$. Now by the induction hypothesis, we have
\begin{eqnarray*}
(V',h,\iota_1)\models\sigma &\iff& (\exists h'\in V') \ [h\equiv_k h' \text{ and } (V',h',\iota_1)\models\sigma']\\
&\iff& (\exists h'\in V) \ [h\equiv_k h' \text{ and } (V',h',\iota_1)\models\sigma']\\
&\iff& (\exists h'\in V) \ [h\equiv_k h' \text{ and } (V,h',\iota)\models\sigma']\\
&\iff& (V,h,\iota)\models\sigma.
\end{eqnarray*}
Similarly, $(V',h,\iota_2)\models\sigma\iff (V,h,\iota)\models\sigma$. We have shown that \eqref{B12} is true. Thus, in particular, we have
\begin{equation}\label{eq}
(V',f,\iota_1)\models\tau\cdot c_0-d_{01} \ \ \text{ and } \ \ (V',f,\iota_2)\models\tau\cdot c_0-d_{01}.
\end{equation}
By the choice of $\iota_1$, we have $(V',h,\iota_1)\models -x_0$ for each $h\not\in W$. Hence, 
\begin{equation}\label{eq1}
(V',f,\iota_1)\models -c_0(-d_{01}\cdot c_i(x_0\cdot -d_{ij})).
\end{equation}
On the other hand, $(V',g(i/g(0)),\iota_2)\models x_0\cdot -d_{ij}$ and $(V',g,\iota_2)\models -d_{01}\cdot c_i(x_0\cdot -d_{ij})$. Whence, it follows that 
\begin{equation}\label{eq2}
(V',f,\iota_2)\models c_0(-d_{01}\cdot c_i(x_0\cdot -d_{ij})).
\end{equation}
Therefore, by \eqref{eq}, \eqref{eq1} and \eqref{eq2}, $\tau\cdot c_0-d_{01}$ is not an atom in $\f$.

\noindent{\textbf{Case 2:}} Suppose that $g(i)\not=g(j)$. Let 
$$W=\{h\in V: h(i)\not=h(j)\} \ \ \text{ and } \ \ V'=V\cup\{g(i/g(j))\}.$$
Again, the assumption on $\u$ guarantees that $V'\in\u$. Define $\iota_1,\iota_2:X\rightarrow\mathcal{P}(V')$ as follows: For each $x_k\in X$,
$$\iota_1(x_k)=\iota(x_k)\cap W \ \ \text{ and } \ \ \iota_2(x_k)=\iota_1(x_k)\cup\{g(i/g(j))\}.$$
Similarly to the above case, cf. \eqref{B12}, one can easily show that 
\begin{equation}\label{2eq}
(V',f,\iota_1)\models \tau\cdot c_0-d_{01} \ \ \text{ and } \ \ (V',f,\iota_2)\models \tau\cdot c_0-d_{01}.
\end{equation}
Moreover, the choice of $\iota_1$ and $\iota_2$ implies
\begin{equation}\label{2eq2}
(V',f,\iota_1)\models -c_0(-d_{01}\cdot c_i(x_0\cdot d_{ij})) \ \ \text{ and } \ \ (V',f,\iota_1)\models c_0(-d_{01}\cdot c_i(x_0\cdot d_{ij})).
\end{equation}
Therefore, again by \eqref{2eq} and \eqref{2eq2}, $\tau\cdot-d_{01}$ is not an atom in $\f$.
\end{proof}
Now, we are ready to prove the main result of this paper.
\begin{proof}[Proof of Theorem~\ref{realmain}]
Let $\u$ be a class of $\alpha$-dimensional units and let $m\geq 1$ be a finite cardinal. 
\begin{enumerate}
\renewcommand{\theenumi}{(\arabic{enumi})}
\renewcommand{\labelenumi}{\theenumi}
\item If $\u$ supports diagonalization, then Proposition~\ref{prop} implies that $\f$ is not atomic. 
\item Suppose that $\u$ requires diagonalization. By Lemma~\ref{lem}, Lemma~\ref{lem2} and Proposition~\ref{prop}, we have shown that $\f$ contains exactly $2^m$-many atoms, each of which is zero-dimensional. Let $\a{A}\subseteq\f$ be the subalgebra generated by $\{a\cdot -c_0-d_{01}:a\in\f\}$, and let $\a{B}\subseteq \f$ be the subalgebra generated by $\{a\cdot c_0-d_{01}:a\in\f\}$. It is not hard to see that $a\mapsto (a\cdot-c_0-d_{01},a\cdot c_0-d_{01})$ is an isomorphism from $\f$ onto $\a{A}\times\a{B}$. Clearly, $\a{A}$ and $\a{B}$ satisfy the desired of item (c).
\item Suppose that $\u$ is the class of all $\alpha$-dimensional units. Let $Y$ be the generating set of the free algebra $\fz$. Let $\tau\in T_{\alpha}(Y)$ be any term such that $\cu\not\models\tau=0$. We will show that $\tau$ is not an atom in $\fz$.

Let $V\in\u$, $f\in V$ and $\iota:Y\rightarrow\mathcal{P}(V)$ be an evaluation such that $(V,f,\iota)\models\tau$. Let $\Gamma= index(\tau)$ and let $i,j \in\alpha\sim\Gamma$ be such that $i\not=j$. Pick brand new elements $a,b$ that are not in the base of $V$ such that $a = b\iff f(i)\not=f(j)$. For every $h\in V$ with $h\equiv_{\Gamma} f$, let $h^{*}$ be the sequence given as follows: $h^{*}(i) = a$, $h^{*}(j) = b$ and $h^{*}(k) = h(k)$, for every $k\in\alpha\sim\{i, j\}$. Set $V^{*} = \{h^{*} : h \in V \text{ and } h\equiv_{\Gamma}f \}$. Define the evaluation $\iota^{*}:Y\rightarrow\mathcal{P}(V)$ as follows. For each $y\in Y$, let $\iota^{*}(y)=\{h^{*}\in V^{*}: h\in\iota(y)\}$. 
We are going to show that for every $\sigma\in T_{\alpha}(Y)$ and every $h\in V$, if $index(\sigma)\subseteq\Gamma$ and $h\equiv_{\Gamma}f$ then
\begin{equation}\label{B13}
(V,h,\iota)\models\sigma\iff(V^{*},h^{*},\iota^{*})\models\sigma.
\end{equation}
This can be shown by an induction on the complexity of the term $\sigma$ as follows. Obviously, \eqref{B13} is true for the case when $\sigma=x_k\in X$ and when $\sigma=d_{k\lambda}$, $k,\lambda\in\Gamma$. Again, it is not hard to see that the induction step holds for the Boolean connectives. Now, suppose that $\sigma=c_k\sigma'$ and $index(\sigma)\subseteq\Gamma$. That means $k\in\Gamma$ and $index(\sigma')\subseteq\Gamma$, too. Let $h\in V$ be such that $h\equiv_{\Gamma}f$. By the induction hypothesis, we have
\begin{eqnarray*}
(V^{*},h^{*},\iota^{*})\models\sigma &\iff& (\exists g^{*}\in V^{*}) \ [g^{*}\equiv_k h^{*} \text{ and } (V^{*},g^{*},\iota^{*})\models\sigma']\\
&\iff& (\exists g\in V) \ [g\equiv_k h \text{ and } (V,g,\iota)\models\sigma']\\
&\iff& (V,h,\iota)\models\sigma.
\end{eqnarray*}
We have shown that \eqref{B13} is true. Thus, in particular, we have $(V^{*} , f^{*} ,\iota^{*})\models\tau$ . But, by the choice of $a$ and $b$, we have
\begin{equation}
(V,f,\iota)\models d_{ij}\iff (V^{*},f^{*},\iota^{*})\models -d_{ij}.
\end{equation}
Therefore, both $\tau\cdot d_{ij}$ and $\tau\cdot-d_{ij}$ are non-zero in the free algebra $\fz$, i.e., $\tau$ is not an atom in $\fz$, as desired.\qedhere
\end{enumerate}
\end{proof}

By the argument we used in (Fact 3), see page \pageref{fact3} herein, we know that each of the free algebras $\mathfrak{Fr}_0\mathrm{CA}_{\alpha}$ and $\mathfrak{Fr}_0\mathrm{Gs}_{\alpha}$ contains exactly $2^0=1$ atom which happens to be zero-dimensional. This argument cannot be used to obtain similar results for the $0$-generated free algebras of the classes $\mathrm{D}_{\alpha}$ and $\mathrm{G}_{\alpha}$ because none of these is locally finite dimensional. Moreover, our method here to obtain the results concerning these classes depends essentially on the assumption $m\geq 1$, see Proposition~\ref {prop}.
\begin{prob}
Are there any non-zero-dimensional atoms in $\mathfrak{Fr}_0\mathrm{D}_{\alpha}$ or in $\mathfrak{Fr}_0\mathrm{G}_{\alpha}$? Is any of  $\mathfrak{Fr}_0\mathrm{D}_{\alpha}$ and $\mathfrak{Fr}_0\mathrm{G}_{\alpha}$ atomic?
\end{prob}
\section{On Tarski's conjecture}
Now, we prove Theorem~\ref{nemadd}. This theorem shows a difference between $\mathfrak{Fr}_m\mathrm{Gs}_{\alpha}$ and  $\mathfrak{Fr}_m\mathrm{CA}_{\alpha}$ and points to the direction that Tarski's conjecture \cite[Remark 2.5.12]{hmt1} might fail. We will use several notions from the theory of cylindric algebras, e.g., generalized cylindrifications \cite[Definition~1.7.1]{hmt1}, substitutions \cite[Definition~1.5.1]{hmt1}, reducts of $\mathrm{CA}_{\alpha}$'s,  \cite[Definition 2.6.1]{hmt1} and neat reducts of $\mathrm{CA}_{\alpha}$'s \cite[2.6.28]{hmt1}. For instance, for each $x,y\in \mathfrak{Fr}_m\mathrm{CA}_{\alpha}$ and each $i,j\in\alpha$ with $i\not=j$, we let
$$x\oplus y\myeq (x\cdot -y)+(-x\cdot y), \ \ \ c_{(2)}x\myeq c_0c_1x \ \ \text{ and } \ \ s^i_jx\myeq c_i(x\cdot d_{ij}).$$
\begin{proof}[Proof of Theorem~\ref{nemadd}] Remember that $m\geq 1$. Let $x$ be one of the free generators of $\mathfrak{Fr}_m\mathrm{CA}_{\alpha}$. We will define a $\mathrm{CA}_{\alpha}$-term $\tau(x)$ with the desired property as follows: $\tau(x)\myeq x\cdot\chi(x)$, where 
$$\chi(x)\myeq
-c_{(2)}(c_0x\oplus c_0y)-c_{(2)}(c_1x\oplus c_1y)-c_{(2)}[c_1(d_{01}\cdot c_0x)\cdot c_0x-d_{01}]-c_{(2)}[c_0(d_{01}\cdot c_1x)\cdot c_1x-d_{01}],$$
and $y\myeq c_0x\cdot c_1x-x$. Clearly, $\tau(x)\in\mathfrak{Fr}_m\mathrm{CA}_{\alpha}$. Now we prove that $\tau(x)\not=0$ in $\mathfrak{Fr}_m\mathrm{CA}_{\alpha}$\footnote{We note that $\tau(x)=0$ in $\mathrm{Gs}_{\alpha}$ by Theorem~\ref{main}.}.

\begin{cl}\label{firstclaim}
$\tau(x)\not=0$ in $\mathfrak{Fr}_m\mathrm{CA}_{\alpha}$.
\end{cl}
\begin{proof}
It suffices to construct an $\mathfrak{A}\in\mathrm{CA}_{\alpha}$ such that $\tau^{\mathfrak{A}}(a)\not=0$ for some $a\in A$. Let $p\in{^{\alpha}\alpha}$ be the identity sequence defined by: $p(i)=i$ for each $i\in\alpha$. Let $B={^{\alpha}\alpha}\cup\{p'\}$ for some $p'\not\in{^{\alpha}\alpha}$. Let $h:B\rightarrow{^{\alpha}\alpha}$ be defined by $h(q)=q$ if $q\in{^{\alpha}\alpha}$ and $h(p')=p$. Let $i,j\in\alpha$ be such that $i\not=j$ and let $X\subseteq B$. Define,

$\begin{array}{rl}
d_{ii}=&B,\\
d_{ij}=& \{q\in{^{\alpha}\alpha}:q(i)=q(j)\},\\
c_iX=&\{q\in B: (\exists t\in X) \ h(q)\equiv_ih(t)\}.
\end{array}$

We construct $\mathfrak{A}$ as follows: $\mathfrak{A}=\langle \mathcal{P}(B),\cup,\cap,\sim,\emptyset, B,c_i,d_{ij} \rangle_{i,j<\alpha}$. It is easy to check that $\mathfrak{A}$ satisfies the postulates \ref{ca0}-\ref{ca7} of Definition~\ref{ca}. Therefore, $\mathfrak{A}\in\mathrm{CA}_{\alpha}$. 

Let $a\myeq\{p\}$. Then $b\myeq c_0a\cdot c_1a-a=\{p'\}$. It can be checked that $c_0a=c_0b$, $c_1a=c_1b$, $c_1(d_{01}\cdot c_0a)\cdot c_0a\leq d_{01}$ and $c_0(d_{01}\cdot c_1a)\cdot c_1a\leq d_{01}$, hence $\chi^{\mathfrak{A}}(a)=1$, and $\tau^{\mathfrak{A}}(a)=a\not=0$.
\end{proof}
\begin{cl}
Let $i,j\in \alpha\sim 2$ be such that $i\not=j$. Then $\tau(x)\leq -d_{ij}$ in $\mathfrak{Fr}_m\mathrm{CA}_{\alpha}$.
\end{cl}
\begin{proof}
Consider the following system $E(X,Y)$ of equations:

$\begin{array}{ll}
&X\cdot y=0, \ X\not=0, \\
&c_iX=c_iY \ \text{ for }i\in 2,\\
&c_i(d_{01}\cdot c_kX)\cdot c_kX\leq d_{01} \ \text{ for }\{i,k\}=2.
\end{array}$

Let $\eta\myeq\eta(x)\myeq y\cdot\chi(x)$, $\chi\myeq\chi(x)$ and $\tau\myeq\tau(x)$. First we show that $E(\tau,\eta)$ holds in $\mathfrak{Fr}_m\mathrm{CA}_{\alpha}$. 
\begin{enumerate}
\renewcommand{\theenumi}{(\Alph{enumi})}
\renewcommand{\labelenumi}{\theenumi}
\item $\tau\cdot\eta=0$ since $\tau\leq x$ and $\eta\leq -x$.
\item $\tau\not=0$ by Claim~\ref{firstclaim}.
\item Let $i\in 2$. By $c_{(2)}\chi=\chi$, we have $c_i\tau=c_i(x\cdot c_{(2)}\chi)=c_ix\cdot\chi$ and similarly $c_i\eta=c_iy\cdot\chi$, hence $c_i\tau\oplus c_i\eta=(c_ix\oplus c_iy)\cdot \chi=0$ since $\chi\leq -c_{(2)}(c_ix\oplus c_iy)$. This implies $c_i\tau=c_i\eta$. 
\item Let $i,k\in\alpha$ be such that $\{i,k\}=2$. By $c_{(2)}\chi=\chi$, we again have 
$$c_i(d_{01}\cdot c_k\tau)\cdot c_k\tau=[c_i(d_{01}\cdot c_kx)\cdot c_kx]\cdot\chi\leq d_{01}$$
by $\chi\leq -c_{(2)}(c_i(d_{01}\cdot c_kx)\cdot c_kx-d_{01})$.
\end{enumerate}
Let $i,j\in \alpha\sim 2$ be such that $i\not=j$. Let $s^i_j\tau\myeq c_i(\tau\cdot d_{ij})$ and $s^i_j\eta\myeq c_i(\eta\cdot d_{ij})$. Assume that $\tau\cdot d_{ij}\not=0$. Then $s^i_j\tau\not=0$. Now, by $\{i,j\}\cap 2=\emptyset$ and by \cite[Section 1.5]{hmt1}, we have that $E(s^i_j\tau,s^i_j\eta)$ also holds in $\mathfrak{Fr}_m\mathrm{CA}_{\alpha}$. Let $\mathfrak{R}\myeq\mathfrak{Rd}_{2\cup\{i\}}\mathfrak{Fr}_m\mathrm{CA}_{\alpha}$ be the reduct of $\mathfrak{Fr}_m\mathrm{CA}_{\alpha}$ resulting by ignoring the operations that contain indices in $\alpha\sim(2\cup\{i\})$. Consider the neat reduct $\mathfrak{C}=\mathfrak{Nr}_2\mathfrak{R}$. Then, $\mathfrak{C}\in\mathrm{Gs}_2$ by \cite[Theorem 3.2.65]{hmt2}. Let $\tau'\myeq s^i_j\tau$ and $\eta'\myeq s^i_j\eta$. Then $\tau',\eta'\in C$ and $E(\tau',\eta')$ holds in $\mathfrak{C}$. This is a contradiction since $\mathfrak{C}\in\mathrm{Gs}_2$ and it is not diificult to verify that $E(X,Y)$ fails in $\mathfrak{C}\in\mathrm{Gs}_2$ for every $X,Y\in C$. That means our assumption $\tau\cdot d_{ij}\not=0$ cannot hold, i.e., $\tau\leq -d_{ij}$.
\end{proof}
Therefore, there is $x\in\a{Fr}_m\mathrm{CA}_{\alpha}$ such that $x\not=0$ and $x\leq -d_{ij}$ for every $i,j\in\alpha\sim 2$ with $i\not=j$. We note that this proof works to prove Theorem~\ref{nemadd} if $\alpha$ is any arbitrary ordinal, but the theorem is interesting only for the case $\alpha\geq\omega$.
\end{proof}


\begin{thebibliography}{99}
\bibitem{hmt1}
L. Henkin, J. D. Monk and A. Tarski (1971).
\newblock Cylindric Algebras Part I. 
\newblock  vol 64 of studies in logic and the foundation of mathematics.
\newblock North Holland.
\bibitem{hmt2}
L. Henkin, J. D. Monk and A. Tarski (1985).
\newblock Cylindric Algebras Part II. 
\newblock vol 115 of studies in logic and the foundation of mathematics.
\newblock North Holland.
\bibitem{nem86}
I. N{\'e}meti (1986).
\newblock Free algebras and decidability in algebraic logic.
\newblock Academic Doctoral Dissertation (in Hungarian), Hungarian Academy of Sciences, Budapest.
\newblock \url{http://www.renyi.hu/~nemeti/NDis/NDis86.pdf}.
\bibitem{And88}
H. Andr\'eka and R. J. Thompson (1988).
\newblock A Stone type representation theorem for algebras of relations of higher rank. 
\newblock Transactions of the American Mathematical Society 309, pp. 671-682.

\bibitem{andjonnem}
H. Andr\'eka, B. J\'onsson and I N\'emeti (1991).
\newblock Free algebras in discriminator varieties.
\newblock Algebra Universalis, 28, pp. 401--447.

\bibitem{nem95}
I. N\'emeti (1995). 
\newblock Decidable versions of first order logic and cylindric-relativized set algebras. 
\newblock In: Logic colloquium’92, Eds: L. Cisirma, D. M. Gabbay and M. de Rijke, CSLI Publications, Stanford, California, and European Association for Logic, Language and Information, pp.177-241.

\bibitem{nem85}
I. N\'emeti (1985).
\newblock Cylindric-relativized set algebras have strong amalgamation.
\newblock J. Symbolic Logic 50, pp. 689-700.
\bibitem{mikulas98}
Sz. Mikul\'as (1998).
\newblock Taming first-order logic. 
\newblock Logic Journal of the IGPL 6 (2), pp. 305-316.
\bibitem{monk}
J. D. Monk (2000).
\newblock An introduction to cylindric set algebras.
\newblock Logic Journal of the IGPL, 8 (4), pp. 451-492.
\bibitem{And01}
H. Andr\'eka (2001).
\newblock A finite axiomatization of locally square cylindric-relativized set algebras. 
\newblock Studia Sci. Math. Hungar.  38, pp. 1-11.

\bibitem{madnem}
J. Madar\'asz and I. N\'emeti (2001).
\newblock Free Boolean algebras with closure operators and a conjecture of Henkin, Monk, and Tarski. 
\newblock Studia Sci. Math. Hungar. 
38, pp. 273-278.
\bibitem{marx2002}
M. J. Marx (2002).
\newblock Computing with cylindric modal logics and arrow logics, lower bounds. 
\newblock Studia Logica 72 (2), pp. 233-252.
\bibitem{zalan}
Z. Gyenis (2011).
\newblock On atomicity of free algebras in certain cylindric-like varieties.
\newblock Logic Journal of the IGPL, 19 (1), pp. 44-52.
\bibitem{andnem2013}
H. Andr\'eka and I. N\'emeti (2013).
\newblock Reducing First-order Logic to $\mathrm{Df}_3$, free algebras.
\newblock In: H. Andr\'eka, M. Ferenczi and I. N\'emeti, editors (2013). Cylindric-like Algebras and Algebraic Logic. vol 22 of Bolyai Society Mathematical Studies, Springer.
\bibitem{myphd}
M. Khaled (2016).
\newblock G\"odel's incompleteness properties and the guarded fragment: An algebraic approach.
\newblock PhD thesis.
\newblock Central European University. \url{https://mathematics.ceu.edu/sites/mathematics.ceu.hu/files/attachment/basicpage/27/phdthesis.pdf}.
\bibitem{myigpl}
M. Khaled (2017).
\newblock The free non-commutative cylindric algebras are not atomic.
\newblock Logic journal of the IGPL, 25 (5), pp. 673–685.
\bibitem{apal}
M. Khaled (2018).
\newblock First order logic without equality on relativized semantics.
\newblock Annals of Pure and Applied Logic, 169 (11), pp. 1227-1242.
\bibitem{aim}
M. Khaled (2018).
\newblock The finitely axiomatizable complete theories of non-associative arrow frames.
\newblock Advances in Mathematics, 346, pp. 194-218.
\end{thebibliography}
\end{document}